\documentclass[10pt,reqno]{amsart}
\usepackage{amssymb,amscd}   

\newtheorem{theorem}{Theorem}[section]
\newtheorem{proposition}[theorem]{Proposition} 
\newtheorem{corollary}[theorem]{Corollary}
\newtheorem{lemma}[theorem]{Lemma}
\newtheorem{remark}[theorem]{Remark}

\begin{document}
\title[On the predual of non-commutative $H^\infty$]{On the predual of non-commutative $H^\infty$}
\author[Y. Ueda]
{Yoshimichi UEDA}
\address{
Graduate School of Mathematics, 
Kyushu University, 
Fukuoka, 819-0395, Japan
}
\email{ueda@math.kyushu-u.ac.jp}
\thanks{$^*\,$Supported by Grant-in-Aid for Scientific Research (C) 20540213.}
\dedicatory{Dedicated to Professor Yoshikazu Katayama on the occasion of his 60th birthday}

\maketitle

\begin{abstract} 
The liftability property of weakly relatively compact subsets in $M_\star/A_\perp$ to $M_\star$ is established for any non-commutative $H^\infty$-algebra $A = H^\infty(M,\tau)$. Some supplementary results to our previous works are also given. 
\end{abstract}

\allowdisplaybreaks{

\section{Introduction} 

Let $H^\infty(\mathbb{D})$ be the Banach algebra of all bounded analytic functions on the open unit disk $\mathbb{D}$ equipped with the supremum norm $\Vert-\Vert_\infty$, and it is known that $H^\infty(\mathbb{D})$ is faithfully embedded into $L^\infty(\mathbb{T})$ by taking non-tangental limits. Via the embedding $H^\infty(\mathbb{D})$ has the `standard' predual $L^1(\mathbb{T})/H^\infty(\mathbb{D})_\perp$, where $H^\infty(\mathbb{D})_\perp$ denotes the pre-annihilator of $H^\infty(\mathbb{D})$ in the dual pairing $L^1(\mathbb{T})^\star = L^\infty(\mathbb{T})$. The space  $L^1(\mathbb{T})/H^\infty(\mathbb{D})_\perp$ has received much attention in Banach space theory, and indeed many serious investigations were carried out, see e.g.~\cite{Pelczynski:CBMS77},\cite[\S6.d]{Pisier:CBMS86}. The present notes are part of our attempts, started at \cite{Ueda:MathAnn09}, to give more `functional analysis insight' to many theorems  obtained in those investigations on $L^1(\mathbb{T})/H^\infty(\mathbb{D})_\perp$ by discussing them in some non-commutative setup. 

Natural non-commutative generalizations of $H^\infty(\mathbb{D})$ were introduced by Arveson \cite{Arveson:AJM67} in the 60's under the name of maximal subdiagonal algebras, and we here call them non-commutative $H^\infty$-algebras. The finite tracial ones have been well-studied so that we mainly deal with the finite tracial non-commutative $H^\infty$-algebras in the present notes. Let $M$ be a finite von Neumann algebra with a faithful normal tracial state $\tau$. A $\sigma$-weakly closed unital (not necessarily self-adjoint) subalgebra $A$ of $M$ is called a finite tracial non-commutative $H^\infty$-algebra, which we denote by $A = H^\infty(M,\tau)$, if $A+A^*$ is $\sigma$-weakly dense in $M$ and the unique $\tau$-preserving conditional expectation $E : M \rightarrow D := A\cap A^*$ is multiplicative on $A$. Here we denote $A^* = \{ a^* \in M\,\big|\, a \in A \}$, while the symbol $X^\star$ has been used as the dual Banach space of a given Banach space $X$. The reader can find an excellent survey for non-commutative $H^\infty$-algebras in \cite{PisierXu:Handbook03}. It is plain to see that $A$ has the `standard' predual $M_\star/A_\perp$, which is the main object in our study.  

It is the main purpose of the present notes to prove that any weakly relatively compact subset in $M_\star/A_\perp$ can be `lifted up' to a weakly relatively compact subset in $M_\star$. In particular, it immediately follows that the Mackey topology on $A$ is indeed the relative topology induced from that on $M$. Hence this part of the present notes provides a non-self-adjoint generalization of Sakai's result \cite{Sakai:IllinoisJMath65} (also Akemann's result \cite[Theorem II.7]{Akemann:TAMS67}). Here, recall that this liftability property was already established by many hands, e.g., Kisljakov \cite{Kisljakov:SovietMathDokl75}, Delbaen \cite{Delbaen:SeminaireFA-Polytech77-78}, Pe\l czynski \cite[\S7]{Pelczynski:CBMS77}, in the 70's for the classical and commutative $L^1(\mathbb{T})/H^\infty(\mathbb{D})_\perp$ and its function algebra generalizations (see \cite[p.54]{Pelczynski:CBMS77} for further information), and it played a key r\^{o}le in any existing proof of the fact that $L^1(\mathbb{T})/H^\infty(\mathbb{D})_\perp$ has the Dunford--Pettis property. However the predual of any non-type I von Neumann algebra does never have the Dunford--Pettis property due to Bunce \cite{Bunce:PAMS92} so that there is no hope to establish the Dunford--Pettis property for non-commutative $M_\star/A_\perp$ in general. Nevertheless, the liftability of weakly relatively compact subsets still survive for non-commutative $M_\star/A_\perp$. 

The second purpose is to give some supplements to our previous notes \cite{Ueda:MathAnn09}. Firstly we briefly explain how the results in \cite{Ueda:MathAnn09} can easily be generalized to the case of semifinite tracial non-commutative $H^\infty$-algebras $A = H^\infty(M,\mathrm{Tr})$ with semifinite $M$ and faithful normal semifinite tracial weights $\mathrm{Tr}$. Secondly we discuss the results in \cite{Ando:CommentMath78} that still remain to be proved in the non-commutative setup.   

\medskip\noindent
{\it Acknowledgment.} We thank Professor Hermann Pfitzner for communicating us his recent result together with useful comments.  

\section{Weakly Relatively Compact Subsets in $M_\star/A_\perp$} 

Throughout this section let us assume that $A=H^\infty(M,\tau)$ is a finite tracial non-commutative $H^\infty$-algebra. In what follows we freely identify $M_\star/A_\perp$ with `the normal part' of $A^\star$ by the embedding $[\varphi] \in M_\star/A_\perp \mapsto \varphi|_A \in A^\star$. It was already verified in \cite[Corollary 2]{Ueda:MathAnn09} that $M_\star/A_\perp$ has the Pe{\l}czynski property (V$^\star$) so that the weak relative compactness of a given subset $W$ in $M_\star/A_\perp$ is characterized by the following property: For any wuC series $\sum_k a_k$ in $A = (M_\star/A_\perp)^\star$ one has $\lim_{k\rightarrow\infty}\sup\{ |\varphi(a_k)|\ |\ \varphi \in W \} = 0$. Here a wuC (= weakly unconditionally Cauchy) series $\sum_k x_k$ in a Banach space $X$ means a formal series consisting of elements $x_k$ in $X$ such that $\sum_{k=1}^\infty |\phi(x_k)| < +\infty$ for any $\phi \in X^\star$. See \cite[\S2.4]{AlbiacKalton:Book} for more on wuC series. We begin with a lemma, which originates in Chaumat's work \cite[Lemme 2]{Chaumat:Unpublished74},\cite[Lemme 1 and Corollaire 1]{Chaumat:PacificJMath81}.   

\begin{lemma}\label{L-2.1} Let $\{\varphi_k\}$ be a bounded sequence in $M_\star/A_\perp$. Let $\varphi \in A^\star$ be a weak$^\star$--accumulation point of $\{\varphi_k\}$, and $\varphi = \varphi^n + \varphi^s$ be the Lebesgue decomposition in the sense of \cite{Ueda:MathAnn09}, that is, the normal and singular decomposition of its Hahn--Banach extension {\rm(}i.e., norm-preserving extension{\rm)} $\tilde{\varphi} = \tilde{\varphi}^n+\tilde{\varphi}^s$ in the sense of Takesaki \cite[p.126--127]{Takesaki:Book1} provides such a decomposition as simultaneous restrictions to $A$. If $\varphi^s\neq0$, i.e., $\tilde{\varphi}^s|_A \neq 0$, then there is a subsequence $\{\varphi_{k(j)}\}$ such that the mapping $a \in A \mapsto \{\varphi_{k(j)}(a)\} \in \ell^\infty(\mathbb{N})$ is surjective. 
\end{lemma} 
\begin{proof} By assumption there is $b \in A$ with $\tilde{\varphi}^s(b) = 1$, and also by \cite[Theorem 1]{Ueda:MathAnn09} together with Pe{\l}czynski's trick (see the proof of \cite[Lemma 2.9]{HarmandWerner^2:Book} or \cite[Proposition 3.5.2]{AlbiacKalton:Book}) there are a sequence $\{a_i\}$ in $A$ and a projection $p \in M^{\star\star}$ such that (i) $\lim_{i\rightarrow\infty} a_i = 0$ in $\sigma(M,M_\star)$, (ii) $\lim_{i\rightarrow\infty} a_i = p$ in $\sigma(M^{\star\star},M^\star)$, (iii) $\tilde{\varphi}^s = \tilde{\varphi}^s\cdot p$, (iv) $\sum_{i=1}^\infty|\phi(a_i -a_{i-1})| < + \infty$ for every $\phi \in M^\star$ with $a_0 := 0$. Set $b_i := a_i b$, $i \in \mathbb{N}$. Then (a) $\varphi_k(b_i) \rightarrow 0$ as $i \rightarrow\infty$ for every $k \in \mathbb{N}$; (b) $\varphi(b_i) = \tilde{\varphi}(b_i) = \tilde{\varphi}^n(a_i b) + \tilde{\varphi}^s(a_i b) \rightarrow (\tilde{\varphi}^s\cdot p)(b) = \tilde{\varphi}^s(b) = 1$ as $i \rightarrow \infty$. The rest of the proof is the same as in \cite[Lemm\`{e} 2]{Chaumat:Unpublished74} or \cite[Corollaire 1]{Chaumat:PacificJMath81}, but we will give a detailed argument for the sake of completeness. 

We then prove that there are two subsequences $\{b_{i(j)}\}$, $\{\varphi_{k(j)}\}$ with the properties: (c) $|\varphi_{k(j_1)}(b_{i(j_2)})| \leq 1/2^{j_2+2}$ for $j_1 \leq j_2 -1$; (d) $|\varphi_{k(j_1)}(b_{i(j_2)})-1| \leq 1/2^{j_2+2}$ for $j_1 \geq j_2$. By (b) one can find $i(1)$ in such a way that $|\varphi(b_{i(1)})-1| \leq 1/16$. Since $\varphi$ is a weak$^\star$--accumulation point of $\{\varphi_k\}$, there is $k(1)$ so that $|\varphi_{k(1)}(b_{i(1)}) - \varphi(b_{i(1)})| \leq 1/16$. Thus, $|\varphi_{k(1)}(b_{i(1)})-1| \leq 1/8$. Assume, as induction hypothesis,  that one has already constructed the desired $b_{i(j)}$'s and $\varphi_{k(j)}$ until $j=l$, and also that those $b_{i(j)}$, $1 \leq j \leq l$, satisfy $|\varphi(b_{i(j)}) - 1| \leq 1/2^{j+3}$. By (a) and (b) one can find $i(l+1)$ with $i(l) \lneqq i(l+1)$ in such a way that $|\varphi_{k(j)}(b_{i(l+1)})| \leq 1/2^{l+3}$ for all $1 \leq j \leq l$ and $|\varphi(b_{i(l+1)})-1| \leq 1/2^{l+4}$. Since $\varphi$ is weak$^\star$--accumulation point of $\{\varphi_k\}$, one can also find $k(l+1)$ with $k(l) \lneqq k(l+1)$ in such a way that $|\varphi(b_{i(j)}) - \varphi_{k(l+1)}(b_{i(j)})|\leq 1/2^{l+4}$ for all $1 \leq j \leq l+1$. Hence $|\varphi_{k(l+1)}(b_{i(j)})-1| \leq 1/2^{l+3} \leq 1/2^{j+2}$ for all $1 \leq j \leq l+1$. In this way, the desired subsequences can be constructed by induction.  

Let $\alpha = \{\alpha_j\}$ be an arbitrary sequence of complex numbers with $\Vert\alpha\Vert_\infty := \sup|\alpha_j| < +\infty$. Set $c_{\alpha,m} := \sum_{j=1}^m \alpha_j(b_{i(j)} - b_{i(j+1)}) \in A$, $m \in \mathbb{N}$. Notice that the subsequence $\{b_{i(j)}\}$ still satisfies that (e) $\sum_{j=1}^\infty |\phi(b_{i(j+1)} - b_{i(j)})| < +\infty$ for any $\phi \in M^\star$ due to (iv). Thus there is $c_\alpha := \sum_{j=1}^\infty \alpha_j (b_{i(j+1)}-b_{i(j)}) \in M$ in the $\sigma$-weak topology such that $\psi(c_\alpha) = \sum_{j=1}^\infty \alpha_j\psi(b_{i(j)}-b_{i(j+1)})$, $\psi \in M_\star$ and, in particular, $c_\alpha$ falls in $A$. Remark that there is a universal constant $K>0$ (i.e., independent of the particular choice of $\alpha$) such that $\Vert c_\alpha\Vert_\infty \leq K \Vert\alpha\Vert_\infty$ (see the proof of \cite[Lemma 2.4.6]{AlbiacKalton:Book}). By (c) and (d) one has $|\varphi_{k(j_1)}(b_{i(j_2)}-b_{i(j_2+1)})| \leq 1/2^{j_2+1}$ if $j_1 \neq j_2$ and also $|\varphi_{k(j)}(b_{i(j)}-b_{i(j+1)})-1| \leq 1/2^{j+1}$. Hence $|\varphi_{k(j)}(c_\alpha) - \alpha_j| \leq \Vert\alpha\Vert_\infty/2$. Set $d_1 := c_\alpha$ Then, consider $\alpha^{(2)} := \{\alpha_j - \varphi_{k(j)}(d_1)\}$ and construct $d_2 := c_{\alpha^{(2)}}$ from $\alpha^{(2)}$ in the same way as above. Since $\Vert\alpha^{(2)}\Vert_\infty \leq \Vert\alpha\Vert_\infty/2$, one has $\Vert d_2\Vert_\infty \leq K\Vert\alpha\Vert_\infty/2$ and $|\varphi_{k(j)}(d_1+d_2) - \alpha_j| \leq \Vert\alpha\Vert_\infty/4$. In this way, we can inductively construct $d_3,d_4,\dots \in A$ so that $\Vert d_l \Vert_\infty \leq K\Vert\alpha\Vert_\infty/2^{l-1}$ and $|\varphi_{k(j)}(\sum_{l=1}^m d_l) - \alpha_j| \leq \Vert\alpha\Vert_\infty/2^m$. Therefore, $\sum_{l=1}^\infty d_l \in A$ and $\varphi_{k(j)}(\sum_{l=1}^\infty d_l) = \alpha_j$ for all $j$. \end{proof}

\begin{remark}\label{R-2.2} {\rm The above proof gives a direct proof to the part of \cite[Corollary 2]{Ueda:MathAnn09} showing that $M_\star/A_\perp$ has the property (V$^\star$) as follows. Assume that a given bounded subset $W \subset M_\star/A_\perp$ is NOT weakly relatively compact. By the Eberlein--Smulian theorem one can find a sequence $\{\varphi_k\} \subset W$ such that any its subsequence does NOT converge in $\sigma(M_\star/A_\perp,A)$. By the Alaoglu theorem $\{\varphi_k\}$ has a weak$^\star$--accumulation point $\varphi \in A^\star$, and the assumption here says that $\varphi^s \neq 0$. Hence, by the above proof one can find a subsequence $\{\varphi_{k(j)}\}$ and a sequence $\{b_{i(j)}\}$ of elements in $A$ with the properties (c),(d),(e). Then, by (c),(d) one has $|\varphi_{k(j-1)}(b_{i(j)}-b_{i(j-1)})-1| \leq |\varphi_{k(j-1)}(b_{i(j)})| + |\varphi_{k(j-1)}(b_{i(j-1)})-1| \leq 1/2^{j+2} + 1/2^{j+1}$, and thus $1 - |\varphi_{k(j-1)}(b_{i(j)}-b_{i(j-1)})| \leq 1/2^{j+2} + 1/2^{j+1}$. Therefore, $\sup\{|\varphi(b_{i(j)}-b_{i(j-1)})|\,|\,\varphi \in W\} \geq |\varphi_{k(j-1)}(b_{i(j)}-b_{i(j-1)})| \geq 1-1/2^{j+2} - 1/2^{j+1} \rightarrow 1$ as $j\rightarrow\infty$. Moreover, by (e), $\sum_j (b_{i(j)}-b_{i(j-1)})$ is a wuC series.}
\end{remark}

The non-commutative Gleason--Whitney theorem (\cite[Theorem 3 (2)]{Ueda:MathAnn09},\cite[Theorem 4.1]{BlecherLabuschagne:Studia07}, and also see Remark \ref{R-2.5} below) says that the unique Hahn--Banach extension $\varphi \in M_\star/A_\perp (\subset A^\star) \mapsto \tilde{\varphi} \in M_\star$ gives a right inverse of the quotient map from $M_\star$ onto $M_\star/A_\perp$. Let us denote by $(M_\star/A_\perp)^+$ the set of all $\varphi \in M_\star/A_\perp$ with $\varphi(1) = \Vert\varphi\Vert$. It is easy to see that $(M_\star/A_\perp)^+ = \{\psi|_A \in M_\star/A_\perp\,|\,\psi \in M_\star^+ \}$, where $M_\star^+$ denotes all the positive normal linear functionals on $M$.   

\begin{corollary}\label{C-2.3} A subset $W \subseteq (M_\star/A_\perp)^+$ is weakly relatively compact if and only if so is $\tilde{W} := \{ \tilde{\varphi} \in M_\star^+\,\big|\,\varphi \in W \}$.  
\end{corollary} 
\begin{proof} The `if' part is trivial. Hence it suffices to prove the `only if' part. On contrary, suppose that $\tilde{W}$ is NOT weakly relatively compact. By the Eberlein--Smulian theorem there is a sequence $\{\tilde{\varphi}_k\} \subset \tilde{W}$ such that any its subsequence does NOT  converge in $\sigma(M_\star/A_\perp,A)$. Since $W$ is weakly relatively compact by assumption, $W$ is bounded so that $\tilde{W}$ is too. Hence the Alaoglu theorem shows that $\{\tilde{\varphi}_k\}$ has a weak$^\star$--accumulation point $\psi \in M^\star$ whose Lebesgue decomposition $\psi = \psi^n+\psi^s$ (\cite[Ch.~III]{Takesaki:Book1}) must satisfy $\psi^s \neq 0$. Since all $\tilde{\varphi}_k$'s are positive, so is $\psi$ and hence $\psi^s|A \neq 0$ because $\psi^s$ is also positive and $\psi^s(1) = \Vert\psi^s\Vert \neq 0$. Clearly $\psi|_A$ gives a weak$^\star$--accumulation point of $\{\varphi_k\}$ ($\subseteq W$), and Lemma \ref{L-2.1} shows the existence of a subsequence $\{\varphi_{k(j)}\}$ so that $a \in A \mapsto \{\varphi_{k(j)}(a)\} \in \ell^\infty(\mathbb{N})$ is surjective, a contradiction to the weak relative compactness of $W$. (Indeed, the Eberlein--Smulian theorem enables us to find a subsequence $\{\varphi_{k(j(l))}\}$ such that $\varphi_{k(j(l))}(a)$ converges as $l \rightarrow \infty$ for every $a \in A$. However, an element $\alpha = \{\alpha_j\} \in \ell^\infty(\mathbb{N})$ with $\alpha_{j(l)} := (-1)^l$ satisfies that $\alpha_{j(l)}$ does not converge as $l \rightarrow \infty$.)    
\end{proof}  

The next is a technical lemma. A key idea came to me by a conversation with Professor Masamichi Takesaki some years ago.  

\begin{lemma}\label{L-2.4} Let $\tilde{\varphi} \in M_\star$ be the Hahn--Banach extension of a given $\varphi \in M_\star/A_\perp$. Then there exists $a \in A$ with $\Vert a \Vert_\infty \leq 1$ such that $\varphi(a) = \Vert\varphi\Vert = \Vert\tilde{\varphi}\Vert$, and moreover such an element $a \in A$ satisfies that $\tilde{\varphi} = a^*\cdot|\tilde{\varphi}|$ and $|\tilde{\varphi}| = a\cdot\tilde{\varphi}$.  
\end{lemma}
\begin{proof} The existence of such $a \in A$ is clear due to the Alaoglu theorem and the `normality' of $\varphi \in M_\star/A_\perp$. Thus it suffices to show the latter half. 

Let $\tilde{\varphi} = v\cdot|\tilde{\varphi}|$ be the polar decomposition in the sense of Sakai (see \cite[\S4 in Ch.~III]{Takesaki:Book1}). By the Cauchy--Schwarz inequality and $v^* v \leq 1$, $aa^* \leq 1$ one has $\Vert\tilde{\varphi}\Vert = \tilde{\varphi}(a) = |\tilde{\varphi}|(av) \leq |\tilde{\varphi}|(v^*v)^{1/2} |\tilde{\varphi}|(aa^*)^{1/2} \leq |\tilde{\varphi}|(1) = \Vert\tilde{\varphi}\Vert$, and hence      $|\tilde{\varphi}|(av) = |\tilde{\varphi}|(v^*v)^{1/2} |\tilde{\varphi}|(aa^*)^{1/2}$. This implies that $a^*$ and $v$ agree in the GNS Hilbert space associated with $|\tilde{\varphi}|$ so that $a^* - v \in \{x \in M\,|\,|\tilde{\varphi}|(x^*x) = 0 \} = M(1-s(|\tilde{\varphi}|))$, where $s(|\tilde{\varphi}|)$ denotes the support projection of $|\tilde{\varphi}|$. Hence $a^*\,s(|\tilde{\varphi}|) = v\,s(|\tilde{\varphi}|) = v$ thanks to $v^*v= s(|\tilde{\varphi}|)$. Similarly, by using the right polar decomposition $\tilde{\varphi} = (v\cdot|\tilde{\varphi}|\cdot v^*)\cdot v$ one can show that $a\,s(v\cdot|\tilde{\varphi}|\cdot v^*) = v^*$. Then the desired two equations immediately follow.  
\end{proof}

\begin{remark}\label{R-2.5} {\rm The above lemma gives a proof of the uniqueness part of the non-commutative Gleason--Whitney theorem without any non-commutative $L^p$-tool. Let $\psi_1, \psi_2$ be two Hahn--Banach extensions of a given $\varphi \in M_\star/A_\perp$, i.e., $\psi_i \in M^\star$ with $\Vert\psi_i\Vert = \Vert\varphi\Vert$, $i=1,2$. \cite[Theorem 3 (2)]{Ueda:MathAnn09} shows that both $\psi_i$'s fall in $M_\star$. By the above lemma one has $\psi_i = a^*\cdot|\psi_i|$ and $|\psi_i| = a\cdot\psi_i$, $i=1,2$, for common $a \in A$ with $\Vert a\Vert_\infty \leq 1$. Then, for $b \in A$ one has $|\psi_1|(b) = \psi_1(ba) = \varphi(ba) = \psi_2(ba) = |\psi_2|(b)$, and then $|\psi_1|(b^*) = |\psi_2|(b^*)$. Hence $|\psi_1|, |\psi_2|$ agree on $A+A^*$ so that $|\psi_1| = |\psi_2|$ since those linear functionals are normal. Therefore, $\psi_1 = a^*\cdot|\psi_1| = a^*\cdot|\psi_2| = \psi_2$.} 
\end{remark}

In what follows we define the absolute value $|\varphi|$ for a given $\varphi \in M_\star/A_\perp$ as the restriction of $|\tilde{\varphi}|$ to $A$, and write $|W| := \{ |\varphi| \in (M_\star/A_\perp)^+\,\big|\,\varphi \in W \}$ for a given subset $W \subseteq M_\star/A_\perp$. 

\begin{remark}\label{R-2.6} {\rm Corollary \ref{C-2.3}, Lemma \ref{L-2.4} and \cite[Corollary 2]{Ueda:MathAnn09} (also Remark \ref{R-2.2} above) are enough to show the liftability of weakly relatively compact subsets in $L^1(\mathbb{T})/H^\infty(\mathbb{D})_\perp$ or more generally in any `commutative' $M_\star/A_\perp$ as follows. Firstly notice that we need to show only that $|W|$ is weakly relatively compact if so is a given subset $W \subseteq M_\star/A_\perp$ thanks to Corollary \ref{C-2.3} together with Kazuyuki Sait\^{o}'s result \cite[Theorem 1]{Saito:TohokuMathJ67}. On contrary, suppose that $|W|$ is NOT weakly relatively compact. Since $M_\star/A_\perp$ has the property (V$^\star$) due to \cite[Corollary 2]{Ueda:MathAnn09} (also Remark \ref{R-2.2} above), there is a wuC series $\sum_k b_k$ in $A$ such that $\sup\{\big||\varphi|(b_k)\big|\,\big|\,\varphi \in W\}$ does NOT converge to $0$ as $k \rightarrow \infty$. Passing to a subsequence of $\{b_k\}$ we may assume that there are a sequence $\{\varphi_k\}$ ($\subseteq W$) and $\varepsilon>0$ so that (i) $\big||\varphi_k|(b_k)\big| \geq \varepsilon$ for all $k$ and (ii) $\sum_{k=1}^\infty |\phi(b_k)| < +\infty$ for any $\phi \in M^\star$. It is plain to see, by (ii), that there is a universal constant $C>0$ with $\sum_{k=1}^m |\phi(b_k)| \leq C\Vert\phi\Vert$ for any $\phi \in M^\star$ and every $m$. By Lemma \ref{L-2.4} one has $\varphi_k = a_k\cdot|\varphi_k|$ for some $a_k \in A$ with $\Vert a_k\Vert_\infty \leq 1$. Now, let us assume that $M$ is commutative. Via the Gel'fand representation $M = C(\Omega)$ we can easily prove that there is a universal constant $C > 0$ such that $\sum_{k=1}^m |\phi(b_k a_k)| \leq C\Vert\phi\Vert$ for any $\phi \in M^\star$ and every $m$. (We could not prove this without the commutativity assumption.) Hence $\sum_k b_k a_k$ is a wuC series in $A$ so that $\lim_{k\rightarrow\infty}\sup\{|\varphi(b_k a_k)| : \varphi \in W\} = 0$ by a well-known corollary of Schur's theorem (see \cite[Lemma 7.1]{Pelczynski:CBMS77}). This contradicts (i).}
\end{remark}

As saw in Remark \ref{R-2.6} we have to follow a different path from many `classical and commutative' proofs, see e.g.~\cite[\S7]{Pelczynski:CBMS77}. The next lemma is a non-self-adjoint counterpart of \cite[Lemma 2]{Sakai:IllinoisJMath65} and \cite[Lemma II.3b]{Akemann:TAMS67}, and we emphasize that our argument is more involved than those due to the non-self-adjointness. The lemma can also be regarded as a non-commutative counterpart of a result in \cite[\S3]{CnopDelbaen:JFA77} (also in \cite[\S4]{Konig:ArchMath83}), but our argument is technically more  involved due to the non-commutativity. 

\begin{lemma}\label{L-2.7} Let $\{\varphi_k\}$ be a sequence in $M_\star/A_\perp$ with $\varphi_k \rightarrow 0$ in $\sigma(M_\star/A_\perp,A)$. Then, for each $\varepsilon>0$ there exist $\delta>0$ and $k_0 \in \mathbb{N}$ such that one has $|\varphi_k(a)| < \varepsilon$ for all $k \geq k_0$ and all $a \in A$ with $\Vert a \Vert_\infty \leq 1$ and $\Vert a\Vert_{\tau,2} < \delta$, where $\Vert a \Vert_{\tau,2} := \tau(a^* a)^{1/2}$.   
\end{lemma}
\begin{proof} Let $C := \sup\{\Vert\varphi_k\Vert\, \big|\, k \in \mathbb{N}\} < +\infty$ thanks to the uniform boundedness principle. On contrary, we suppose that the desired assertion does NOT hold true. Namely, there are $\{a_j\}$ in the unit ball of $A$, $\varepsilon>0$, and a subsequence $\{\varphi_{k(j)}\}$ such that (i) $\Vert a_j\Vert_{\tau,2} \rightarrow 0$ as $j\rightarrow\infty$ and (ii) $|\varphi_{k(j)}(a_j)| \geq \varepsilon$ for all $j$. 

In the standard representation $M \curvearrowright \mathcal{H} := L^2(M,\tau)$  the $2$-norm $\Vert-\Vert_{\tau,2}$ gives the strong operator topology on the unit ball of $M$ so that the unit ball of $A$ is complete with respect to $\Vert-\Vert_{\tau,2}$. Since $\varphi_{k(j)} \rightarrow 0$ in $\sigma(M_\star/A_\perp,A)$ as $j \rightarrow \infty$, for a given $\varepsilon>0$ the Baire category theorem enables us to find $j_0 \in \mathbb{N}$, $a_0 \in A$ with $\Vert a_0 \Vert_\infty \leq 1$ and $\delta>0$ in such a way that, for each $a \in A$ with $\Vert a\Vert_\infty \leq 1$, one has: $\Vert a- a_0\Vert_{\tau,2} < \delta$ implies $|\varphi_{k(j)}(a)| < \varepsilon/8$ for all $j \geq j_0$. See the second paragraph of the proof of \cite[Lemma 5.5 in Ch.~III]{Takesaki:Book1} for details. 

Define 
\begin{align*}
b_j &:= \big(1+\Vert a_j\Vert_{\tau,2}^{-1/2}(|a_j|^2 +\sqrt{-1}(|a_j|^2)^\sim)\big)^{-1}, \\
c_j &:= \big(1+\Vert a_j\Vert_{\tau,2}^{-1/2}(|a_j^*|^2 +\sqrt{-1}(|a_j^*|^2)^\sim)\big)^{-1},
\end{align*} 
where the `Hilbert transform' of a given $x \in M$ in the sense of \cite{Randrianantoanina:JAustral98},\cite{MarsalliWest:JOT98} is denoted by $x^\sim$, an unbounded operator affiliated with $M$. Here we remark that $\Vert a_j\Vert_{\tau,2} \neq 0$ thanks to (ii). As in the proof of \cite[Theorem 1]{Ueda:MathAnn09} or more precisely by \cite[Lemma 2]{Randrianantoanina:JAustral98}, we observe that the $b_j$ and the $c_j$ are well-defined contractions in $A$. By the non-commutative Riesz theorem \cite[Theorem 1]{Randrianantoanina:JAustral98},\cite[Theorem 5.4]{MarsalliWest:JOT98} one has $\Vert1-b_j\Vert_{\tau,2} \leq \Vert b_j^{-1} - 1\Vert_{\tau,2} \leq 2 \Vert a_j\Vert_{\tau,2}^{-1/2} \Vert |a_j|^2\Vert_{\tau,2} \leq 2 \Vert a_j\Vert_{\tau,2}^{1/2}$ ({\it n.b.}~$\Vert|a_j|\Vert_\infty = \Vert a_j\Vert_\infty \leq 1$). Similarly $\Vert 1-c_j\Vert_{\tau,2} \leq 2 \Vert a_j\Vert_{\tau,2}^{1/2}$. For any $\zeta \in b_j\mathcal{H}$ we have $\Vert a_j\Vert_{\tau,2}^{-1/2}\Vert a_j\zeta\Vert_{\tau,2}^2 = \Vert a_j\Vert_{\tau,2}^{-1/2}(|a_j|^2\zeta|\zeta)_\mathcal{H} \leq |(b_j^{-1}\zeta|\zeta)_\mathcal{H}| \leq \Vert b_j^{-1}\zeta\Vert_{\tau,2}\Vert\zeta\Vert_{\tau,2}$.  Letting $\zeta := b_j\xi$ with arbitrary $\xi \in \mathcal{H}$ we get $\Vert a_j\Vert_{\tau,2}^{-1/2}\Vert a_j b_j\xi\Vert_{\tau,2}^2 \leq \Vert\xi\Vert_{\tau,2}\Vert b_j \xi\Vert_{\tau,2} \leq \Vert\xi\Vert_{\tau,2}$ so that $\Vert a_j b_j \Vert_\infty \leq \Vert a_j\Vert_{\tau,2}^{1/4}$. Similarly $\Vert a_j^* c_j \Vert_\infty \leq \Vert a_j\Vert_{\tau,2}^{1/4}$. For any $\zeta$ in the domain of $b_j^{-1}$ one has $\mathrm{Re}(b_j(b_j^{-1}\zeta)|(1-b_j)(b_j^{-1}\zeta))_\mathcal{H} = \mathrm{Re}\{(\zeta|b_j^{-1}\zeta)_\mathcal{H} - (\zeta|\zeta)_\mathcal{H}\} = \Vert a_j\Vert_{\tau,2}^{1/2}(\zeta||a_j|^2\zeta)_\mathcal{H} \geq 0$ so that $\mathrm{Re}(b_j\xi|(1-b_j)\xi)_\mathcal{H} \geq 0$ for all $\xi \in \mathcal{H}$. Consequently, we have constructed contractions $b_j, c_j \in A$, $j \in \mathbb{N}$, satisfying the following properties: 
\begin{itemize}
\item[(iii)] $\Vert 1 - b_j\Vert_{\tau,2} \leq 2\Vert a_j\Vert_{\tau,2}^{1/2}$, $\Vert 1 - c_j \Vert_{\tau,2} \leq 2\Vert a_j\Vert_{\tau,2}^{1/2}$; 
\item[(iv)] $\Vert a_j b_j\Vert_\infty \leq \Vert a_j\Vert_{\tau,2}^{1/4}$, $\Vert a_j^* c_j\Vert_\infty \leq \Vert a_j \Vert_{\tau,2}^{1/4}$; 
\item[(v)] $\mathrm{Re}(b_j\xi|(1-b_j)\xi)_\mathcal{H} \geq 0$ for all $\xi \in \mathcal{H}$.
\end{itemize}

For any $\xi \in \mathcal{H}$ with $\Vert\xi\Vert_{\tau,2} \leq 1$ we have 
\begin{align*}
&\Vert (2^{-1}a_j(1-b_j) + (1-\delta/2)c_j a_0 b_j)\xi\Vert_\mathcal{H}^2 \\
&\leq 
2^{-2} \Vert(1-b_j)\xi\Vert_\mathcal{H}^2 + (1-\delta/2)\,\mathrm{Re}(a_j(1-b_j)\xi|c_j a_0 b_j\xi)_\mathcal{H} + (1-\delta/2)^2\Vert b_j\xi\Vert_{\tau,2}^2 \\
&\leq \max\{1/2,1-\delta/2\}^2\big(\Vert(1-b_j)\xi\Vert_\mathcal{H}^2 + 2\mathrm{Re}(b_j\xi|(1-b_j)\xi)_\mathcal{H} + \Vert b_j\xi\Vert_{\tau,2}^2\big) \\
&\phantom{aaaaaaaaaaaaaaaaaaaaaaaaaaaaaaaaaaaaaaaaaa}+ |((1-b_j)\xi|a_j^* c_j a_0 b_j\xi)_\mathcal{H}| \\
&\leq \max\{1/2,1-\delta/2\}^2 + 2\Vert a_j^* c_j\Vert_\infty, 
\end{align*} 
since $\mathrm{Re}(b_j\xi|(1-b_j)\xi)_\mathcal{H} \geq 0$ by (v). Hence, by (iv) and (i) one has $\Vert 2^{-1}a_j(1-b_j) + (1-\delta/2)c_j a_0 b_j \Vert_\infty \leq \max\{1/2,1-\delta/2\}^2 + 2\Vert a_j^* c_j\Vert_\infty \longrightarrow \max\{1/2,1-\delta/2\}^2 \lneqq 1$ as $j \rightarrow \infty$. Therefore, $\Vert 2^{-1}a_j(1-b_j) + (1-\delta/2)c_j a_0 b_j \Vert_\infty \leq 1$ for all sufficiently large $j$. By (iii) and (i) one has 
\begin{align*}
\Vert 2^{-1}a_j(1-b_j) &+ (1-\delta/2)c_j a_0 b_j -a_0 \Vert_{\tau,2} \\
&\leq \Vert a_j\Vert_{\tau,2} + (1-\delta/2)\{\Vert c_j a_0 (b_j - 1)\Vert_{\tau,2} + \Vert (c_j-1)a_0\Vert_{\tau,2}\} + \delta/2 \\
&\leq \Vert a_j\Vert_{\tau,2} + 4\Vert a_j\Vert_{\tau,2}^{1/2} + \delta/2 \longrightarrow \delta/2
\end{align*} 
as $j \rightarrow \infty$. Similarly, $\Vert(1-\delta/2)c_j a_0 b_j - a_0\Vert_{\tau,2} \leq 4\Vert a_j\Vert_{\tau,2}^{1/2} + \delta/2 \longrightarrow \delta/2$ as $j \rightarrow \infty$. It follows that $\Vert 2^{-1}a_j(1-b_j) + (1-\delta/2)c_j a_0 b_j -a_0 \Vert_{\tau,2} < \delta$ and $\Vert(1-\delta/2)c_j a_0 b_j - a_0\Vert_{\tau,2} < \delta$ for all sufficiently large $j$. Therefore, by what we prepared in the second paragraph we have, for all sufficiently large $j$, 
\begin{align*}
|\varphi_{k(j)}(2^{-1}a_j)| 
&\leq 
|\varphi_{k(j)}(2^{-1}a_j b_j)| + |\varphi_{k(j)}(2^{-1}a_j(1-b_j) + (1-\delta/2)c_j a_0 b_j)| \\
&\phantom{aaaaaaaaaaaaaaaaaaaaaaaaaaaaaaaaaaa}+ |\varphi_{k(j)}((1-\delta/2)c_j a_0 b_j))| \\
&\leq 
C\Vert a_j\Vert_{\tau,2}^{1/4}/2 + |\varphi_{k(j)}(2^{-1}a_j(1-b_j) + (1-\delta/2)c_j a_0 b_j)| \\
&\phantom{aaaaaaaaaaaaaaaaaaaaaaaaaaaaaaaaaaa}+ |\varphi_{k(j)}((1-\delta/2)c_j a_0 b_j))| \\
&< C\Vert a_j\Vert_{\tau,2}^{1/4}/2 + \varepsilon/8 + \varepsilon/8,
\end{align*} 
since $|\varphi_{k(j)}(a_j b_j)| \leq \Vert \varphi_{k(j)}\Vert \cdot \Vert a_j b_j\Vert_\infty \leq C\Vert a_j \Vert_{\tau,2}^{1/4}$ by (iv). Hence we have $|\varphi_{k(j)}(a_j)| = 2|\varphi_{k(j)}(2^{-1} a_j)| \leq 2\{ C\Vert a_j\Vert_{\tau,2}^{1/4}/2 + \varepsilon/8 + \varepsilon/8\} = C\Vert a_j\Vert_{\tau,2}^{1/4} + \varepsilon/2$ for all sufficiently large $j$ so that by (i) $\limsup_{j\rightarrow\infty}|\varphi_{k(j)}(a_j)| \leq \varepsilon/2$, a contradiction to (ii).       
\end{proof} 

Let $X$ be a Banach space with predual $X_\star$. The Mackey topology on $X$ (with respect to $X_\star$) is the weakest topology that makes all the seminorms $x \in X \mapsto p_W(x) := \sup\{ |\langle x_\star, x\rangle| \,\big|\, x_\star \in W \}$ with weakly relatively compact $W \subseteq X_\star$ be continuous. Namely, a net $\{x_\lambda\}$ converges to $x$ in the Mackey topology if and only if $p_W(x_\lambda-x) \rightarrow 0$ for any weakly relatively compact $W \subset X_\star$. Note that the Mackey topology is clearly stronger than the weak$^\star$--topology $\sigma(X,X_\star)$. 

\begin{corollary}\label{C-2.8} For any weakly relatively compact subset $W \subseteq M_\star/A_\perp$, so is $|W| := \{ |\varphi| \in (M_\star/A_\perp)^+ \,\big|\, \varphi \in W \}$.  
\end{corollary}
\begin{proof} Firstly, note that $|W|$ is bounded since so is $W$. On contrary, suppose that $|W|$ is NOT weakly relatively compact. Since $M_\star/A_\perp$ has Pe\l czynski's property (V$^\star$) due to \cite[Corollary 2]{Ueda:MathAnn09} (also see Remark \ref{R-2.2} in the present notes), there are a sequence $\{\varphi_k\}$ ($\subseteq W$), a wuC series $\sum_k b_k$ in $A$ and $\varepsilon>0$ so that $|\varphi_k|(b_k) \geq \varepsilon$ for all $k$. By Lemma \ref{L-2.4}, $|\varphi_k| = a_k\cdot\varphi_k$ for some $a_k \in A$ with $\Vert a_k\Vert_\infty \leq 1$. It is known, see \cite[Theorem 5.7 in Ch.~III]{Takesaki:Book1}, that the Mackey topology on $M$ and the metric topology by $\Vert-\Vert_{\tau,2}$ agree on every bounded ball. Thus, by a well-known corollary of Schur's theorem (see \cite[Lemma 7.1]{Pelczynski:CBMS77}) we observe that $\Vert b_k\Vert_{\tau,2} \rightarrow 0$ as $k\rightarrow\infty$, and hence $\Vert b_k a_k\Vert_{\tau,2} \rightarrow 0$ as $k\rightarrow\infty$. By the Eberlein--Smulian theorem one can find a subsequence $\{\varphi_{k(j)}\}$ that converges to some $\varphi \in M_\star/A_\perp$ in $\sigma(M_\star/A_\perp,A)$. Since $\sum_k b_k$ is a wuC series and $\Vert a_k\Vert_\infty \leq 1$, one can easily show that $\{b_k a_k\}$ is norm--bounded. Thus, Lemma \ref{L-3.7} shows that $|\varphi_{k(j)}(b_{k(j)}a_{k(j)})-\varphi(b_{k(j)}a_{k(j)})| \rightarrow 0$ as $j\rightarrow\infty$. It follows that $\varepsilon \leq |\varphi_{k(j)}|(b_{k(j)}) = \varphi_{k(j)}(b_{k(j)}a_{k(j)}) \leq |\varphi_{k(j)}(b_{k(j)}a_{k(j)})-\varphi(b_{k(j)}a_{k(j)})| + |\varphi(b_{k(j)}a_{k(j)})| \rightarrow 0$ as $j\rightarrow\infty$, a contradiction. 
\end{proof}    

Here is the main theorem of this section. 

\begin{theorem}\label{T-2.9} For a subset $W \subset M_\star/A_\perp$ the following conditions are equivalent{\rm:} 
\begin{itemize}
\item[(1)] $W$ is weakly relatively compact. 
\item[(2)] $\tilde{W} := \{\tilde{\varphi} \in M_\star \,\big|\, \varphi \in W \}$ is weakly relatively compact.  
\item[(3)] $|W| := \{|\varphi| \in (M_\star/A_\perp)^+ \,\big|\, \varphi \in W \}$ is weakly relatively compact.
\item[(4)] $|\tilde{W}| := \{ |\tilde{\varphi}| \in M_\star^+ \,\big|\, \varphi \in W \}$ is weakly relatively compact. 
\end{itemize} 
\end{theorem} 
\begin{proof} (1) $\Rightarrow$ (3) is Corollary \ref{C-2.8}. (3) $\Leftrightarrow$ (4) follows from Corollary \ref{C-2.3}. (2) $\Leftrightarrow$ (4) is Kazuyuki Sait\^{o}'s result \cite[Theorem 1]{Saito:TohokuMathJ67}. Finally, (2) $\Rightarrow$ (1) is trivial. 
\end{proof} 

The next corollary is immediate from our main theorem. 

\begin{corollary}\label{C-2.10} The Mackey topology on $A$ is the relative topology induced from that on $M$. 
\end{corollary} 

\section{Supplementry Results} 
\subsection{Semifinite Tracial Setup} In \cite{Ueda:MathAnn09} we proved, among other things, that any finite tracial non-commutative $H^\infty$-algebra $A=H^\infty(M,\tau)$ has the unique predual $M_\star/A_\perp$ and also satisfies a kind of F.~\& M.~Riesz theorem which we have called the non-commutative F.~\& M.~Riesz theorem in the present notes. As usual (see \cite[\S4]{Hensgen:MathScand96}), those results can easily be generalized to the semifinite tracial case. 

Let $M$ be a $\sigma$-finite von Neumann algebra with a faithful normal semifinite trace $\mathrm{Tr}$. A semifinite tracial non-commutative $H^\infty$-algebra $A = H^\infty(M,\mathrm{Tr})$ is defined in the same manner with $\mathrm{Tr}$ in place of a faithful normal tracial state $\tau$ as in the finite tracial setup. Due to the hypothesis, the restriction of $\mathrm{Tr}$ to $D := A\cap A^*$ must be semifinite. Thus one can construct, by using e.g.~\cite[Proposition 1.40, Theorem 2.15]{Takesaki:Book1}, at most countable (due to the $\sigma$-finiteness) orthogonal projections $\{e_i\}_{i=1}^\infty$ in $D$ such that $\sum_{i=1}^\infty e_i = 1$ and $\mathrm{Tr}(e_i) < +\infty$ for all $i$. Letting $p_m := \sum_{i=1}^m e_i \in D$ one has $p_m \nearrow 1$ as $m \rightarrow \infty$ and $\mathrm{Tr}(p_m) < +\infty$ for all $m$. The next is a key simple lemma, which shows that questions on semifinite tracial $H^\infty(M,\mathrm{Tr})$ can essentially be reduced to those on finite tracial ones.       

\begin{lemma}\label{L-3.1} For each $m$, $A_m := p_m A p_m$ is a non-commutative $H^\infty$-algebra $H^\infty(M_m,\tau_m)$ associated with $M_m := p_m M p_m$ and $\tau_m := \mathrm
{Tr}(p_m)^{-1}\mathrm{Tr}|_{M_m}$. 
\end{lemma}  
\begin{proof} The restriction of $E$ to $M_m = p_m M p_m$ gives a faithful normal conditional expectation from $M_m$ onto $D_m := A_m \cap A_m^*$ since $p_m D p_m = p_m A p_m \cap p_m A^* p_m$, and also $E_m$ clearly preserves $\tau_m$. It is clear that the restriction of $E_m$ to $A_m$ is multiplicative. Also, for each $x \in M_m$ one has $x \in M$ and $x = p_m x p_m$, and by the hypothesis, there is a net $a_\lambda + b_\lambda^*$ with $a_\lambda, b_\lambda \in A$ such that $a_\lambda + b_\lambda^* \rightarrow x$ $\sigma$-weakly so that $p_m a_\lambda p_m + (p_m b_\lambda p_m)^* = p_m(a_\lambda + b_\lambda^*)p_m \rightarrow p_m x p_m = x$ $\sigma$-weakly. Hence we are done. 
\end{proof} 

With the above lemma and our previous result \cite[Theorem 3.1]{Ueda:MathAnn09}, the proof of \cite[Proposition 4.8]{Hensgen:MathScand96}
works well for $A = H^\infty(M,\mathrm{Tr})$ without any essential change. 

\begin{proposition}\label{P-3.2} $M_\star/A_\perp$ has the property {\rm (X)}, and hence it is the unique predual of $A$. 
\end{proposition} 

The non-commutative F.~\& M.~Riesz theorem also holds true for $A = H^\infty(M,\mathrm{Tr})$.
 
\begin{proposition}\label{P-3.3} Whenever $\phi \in M^\star$ annihilates $A$, the normal and the singular parts $\phi^n$ and $\phi^s$ annihilate $A$ separately. 
\end{proposition} 
\begin{proof} Our previous result \cite[Theorem 3]{Ueda:MathAnn09} implies that $\phi^n|_{A_m} \equiv 0$ and $\phi^s|_{A_m} \equiv 0$ for all $m$ since the normal and singular decomposition of the restriction of $\phi$ to $M_m$ is clearly given by the simultaneous restriction of $\phi = \phi^n + \phi^s$ to $M_m$. By the normality of $\phi^n$ one has $\phi^n(a) = \lim_{m\rightarrow\infty}\phi^n(p_m a p_m) = 0$ for any $a \in A$, and moreover $\phi^s(a) = \phi(a) - \phi^n(a) = 0$ for any $a \in A$.
\end{proof}

The next corollary is immediate from the above proposition and Pfitzner's theorem \cite{Pfitzner:Studia93}. 

\begin{corollary}\label{C-3.4} $M_\star/A_\perp$ is an $L$-embedded space. Hence it has the property {\rm(V$^\star$)}. 
\end{corollary} 

The non-commutative Gleason--Whitney theorem can also be shown for $A = H^\infty(M,\mathrm{Tr})$. The part of `automatic normality' immediately follows from Proposition \ref{P-3.3}, while the uniqueness part is shown in the same manner as in Remark \ref{R-2.5}. 

\begin{corollary}\label{C-3.5} A Hahn--Banach extension of each $\varphi \in M_\star/A_\perp \subseteq A^\star$ to the whole $M$ is automatically normal and unique. Hence the quotient map from $M_\star$ onto $M_\star/A_\perp$ has a right inverse.   
\end{corollary} 

Moreover all the assertions in \cite[Theorem 4]{Ueda:MathAnn09} still hold true in the semifinite tracial case. However we do not know whether or not the liftability of weakly relatively compact subsets in $M_\star/A_\perp$ survives even in the semifinite setting, though Corollary \ref{C-3.5} provides its basic prerequisite fact. Here it should be pointed out that Theorem \ref{T-2.9} does never hold as it is, in the (non-finite) semifinite setting due to Kazuyuki Sait\^{o}'s result \cite{Saito:TohokuMathJ67}. Indeed, the equivalence (1) $\Leftrightarrow$ (3) in Theorem \ref{T-2.9} does not hold in the case of $D= A = M = B(\ell^2)$ for example. In closing of this subsection we mention that it is an interesting question whether the results in \cite{Ueda:MathAnn09} and in the present notes still hold true for any (not necessarily tracial) non-commutative $H^\infty$-algebra.  

\subsection{No Second Predual and Sequential Continuity of $L$-Projection} 

There are two properties on $H^\infty(\mathbb{D})$ established in Ando's paper \cite{Ando:CommentMath78} which remain to be discussed in the non-commutative setup. One is that $H^\infty(\mathbb{D})$ has no second predual (\cite[Corollary of Theorem 1]{Ando:CommentMath78}) and the other is the weak$^\star$-weak--sequential continuity of the $L$-projection from $H^\infty(\mathbb{D})^\star$ onto $L^1(\mathbb{T})/H^\infty(\mathbb{D})_\perp$ (\cite[Theorem 2]{Ando:CommentMath78}). Those still hold true for any semifinite tracial $A=H^\infty(M,\mathrm{Tr})$.     

For the first properties it suffices to prove that the closed unit ball of the unique predual has no extreme point since the closed unit ball of the dual of any Banach space has a rich set of extreme points thanks to the Alaoglu and the Krein--Milman theorems. We begin with a simple lemma. A definitive result in the direction has been known (see e.g.~\cite[Corollary 4.2 + Corollary 4.4]{GuerreroMartin:JFA05}), but we do give a sketch for the reader's convenience because the work \cite{GuerreroMartin:JFA05} deals with it in the framework of $JBW^*$-triples. 

\begin{lemma}\label{L-3.6} If a semifinite von Neumann algebra $M$ is diffuse {\rm(}i.e., no minimal projection{\rm)}, then the closed unit ball of the predual $M_\star$ has no extreme point so that $M$ has no second predual.  
\end{lemma} 
\begin{proof} (Sketch) Identify $M_\star$ with the non-commutative $L^1$-space $L^1(M,\mathrm{Tr})$ equipped with the $1$-norm $\Vert\cdot\Vert_{\mathrm{Tr},1}$, see \cite[\S1]{PisierXu:Handbook03}. Let $h \in L^1(M,\mathrm{Tr})$ be a non-zero element, and $h = v|h|$ be the polar decomposition. It is not hard to see that the relative commutant of the spectral measure of $h$ in $M$ is still diffuse, and hence one can find two non-zero projection $p_1, p_2$ in the relative commutant in such a way that $p_1+p_2$ coincides with the support projection of $h$. Set $h_i := v p_i |h| \in L^1(M,\mathrm{Tr})$, $i=1,2$, and one has $\Vert h\Vert_{\mathrm{Tr},1} = \mathrm{Tr}(|h|) = \mathrm{Tr}(p_1|h|) + \mathrm{Tr}(p_2 |h|) = \Vert h_1\Vert_{\mathrm{Tr},1} + \Vert h_2\Vert_{\mathrm{Tr},1}$ since $p_1,p_2$ commute with $|h|$ in a suitable sense. It immediately follows that the unit ball of $M_\star = L^1(M,\mathrm{Tr})$ does never have an extreme point. 
\end{proof}

The next is a non-commutative counterpart of the fact that $H_0^1(\mathbb{D})$ is an $L$-embedded space (see e.g.~\cite[Example 1.1 (d) in Ch.~IV.1]{HarmandWerner^2:Book}). 

\begin{lemma}\label{L-3.7} The pre-annihilator $A_\perp$ of $A = H^\infty(M,\mathrm{Tr})$ inside $M_\star$ is an $L$-embedded space. 
\end{lemma} 
\begin{proof} Let $\{\phi_\lambda\}$ be a net in $A_\perp$ that converges to $\phi \in M^\star$ in $\sigma(M_\star,M)$. Clearly, $\phi|_A \equiv 0$ holds. Then $\phi$ is decomposed into the normal and singular part $\phi = \phi^n + \phi^s \in M_\star\oplus (M^\star\ominus M_\star)$ in the sense of Takesaki (see \cite[Ch.~III]{Takesaki:Book1}). The non-commutative F.~\& M.~Riesz theorem, i.e, Proposition \ref{P-3.3} shows that $\phi^n|_A \equiv 0$ and $\phi^s|_A \equiv 0$. Then, by Li's criterion (see e.g.~\cite[Theorem 1.2 in Ch.~IV]{HarmandWerner^2:Book}) we get the desired assertion. 
\end{proof} 

Combining the above two lemmas with \cite[Proposition 1.12, Proposition 1.14 in Ch.~IV]{HarmandWerner^2:Book} we get:  

\begin{corollary}\label{C-3.8} The closed unit ball of the unique predual $M_\star/A_\perp$ of $A = H^\infty(M,\mathrm{Tr})$ has no extreme point if $M$ is diffuse. Thus $A$ has no second predual under the same hypothesis. 
\end{corollary}  
 
We should mention that Lemma \ref{L-3.7} gives a more application. In fact, it follows that $A$ and $M_\star/A_\perp$ have the Daugavet property, see \cite[\S2]{GuerreroMartin:JFA05}.    

\medskip 
Next we discuss the weak$^\star$-weak--sequential continuity of the $L$-projection from $A^\star$ onto its $L$-summand $M_\star/A_\perp$ with $A = H^\infty(M,\mathrm{Tr})$. We could give a direct proof to it in the finite tracial case by using only our Amar--Lederer type result \cite[Theorem 1]{Ueda:MathAnn09} together with Pe{\l}czynski's trick, see Remark \ref{R-3.10} below. After the previous version of these notes was circulated via the internet, Prof.~Pfitzner communicated us that he could recently prove the weak$^\star$-weak sequential continuity of the $L$-projection for {\it any} $L$-embedded space. Thus, his result \cite{Pfitzner:ArchivderMath} and Corollary \ref{C-3.4} imply:   

\begin{corollary}\label{C-3.9} The $L$-projection from $A^\star$ onto $M_\star/A_\perp$ with $A = H^\infty(M,\mathrm{Tr})$ is weak$^\star$-weak--sequentially continuous. In particular, the singular part $A^\star\ominus(M_\star/A_\perp)$ is weak$^\star$--sequentially complete. 
\end{corollary}

\begin{remark}\label{R-3.10}{\rm We give a direct proof to Corollary \ref{C-3.9} in the finite tracial case $A = H^\infty(M,\tau)$, though it has a restricted merit now. Our proof seems simpler than Ando's that is a bit tricky and cannot be applied to the non-commutative setting directly. Let $\{\phi_k\}$ be a sequence in $A^\star$, and suppose that $\lim_{k\rightarrow\infty}\phi_k = 0$ in $\sigma(A^\star,A)$. Let $\tilde{\phi}_k$ be the Hahn--Banach extension of $\phi_k$ to the whole $M$, and $\tilde{\phi}_k = \tilde{\phi}_k^n + \tilde{\phi}_k^s$ be the normal and singular decomposition (\cite[Ch.~III]{Takesaki:Book1}). It suffices to prove that $\lim_{k\rightarrow\infty}\tilde{\phi}_k^s(a) = 0$ for all $a \in A$. We may and do assume that all $\tilde{\phi}_k^s$ are non-zero, and define a non-zero positive $\omega := \sum_{k=1}^\infty\frac{1}{2^k\Vert\tilde{\phi}_k^s\Vert}|\tilde{\phi}^s_k| \in M^\star\ominus M_\star$, where $|\tilde{\phi}^s_k|$ is the absolute value of $\tilde{\phi}_k^s$ in the polar decompostion of $\tilde{\phi}_k^s$ (\cite[Ch.~III]{Takesaki:Book1}) with regarding $M^\star$ as the predual of the second dual $M^{\star\star}$. As in Lemma \ref{L-2.1} there are a sequence $\{a_i\}$ in $A$ and a projection $p \in M^{\star\star}$ such that (i) $\lim_{i\rightarrow\infty} a_i = 0$ in $\sigma(M,M_\star)$, (ii) $\lim_{i\rightarrow\infty} a_i = p$ in $\sigma(M^{\star\star},M^\star)$, (iii) $\langle\omega,p\rangle = \omega(1)$, (iv) $\sum_{i=1}^\infty|\phi(a_i -a_{i-1})| < + \infty$ for every $\phi \in M^\star$ with $a_0 := 0$. In particular, $\tilde{\phi}_k^s = \tilde{\phi}_k\cdot p$ holds for every $k$. Choose and fix an arbitrary $a \in A$. Letting $u_j := a_j - a_{j-1}$ we have, by (ii), $\tilde{\phi}_k^s(a) = \langle\tilde{\phi}_k,pa\rangle = \lim_{i\rightarrow\infty}\tilde{\phi}_k(a_i a) = \sum_{j=1}^\infty\tilde{\phi}_k(u_j a)$. For any subset $J \subseteq \mathbb{N}$ the above (iv) ensures the existence of $\sum_{j \in J} u_j a \in A$ such that $\psi(\sum_{j \in J} u_j a) = \sum_{j\in J}\psi(u_j a)$ for any $\psi \in M_\star$ as in Lemma \ref{L-2.1}. Define $\nu_k(J) := \tilde{\phi}_k(\sum_{j\in J} u_j a)$, $J \subseteq \mathbb{N}$, which is a finitely additive signed measure on $\mathbb{N}$ with finite total variation and $\lim_{k\rightarrow\infty}\nu_k(J)=0$ for every $J \subseteq \mathbb{N}$. Thus the classical Phillips lemma (see \cite[Theorem 1.28 in Ch.~III]{Takesaki:Book1}) implies that $\lim_{k\rightarrow\infty}\tilde{\phi}_k^s(a) = \lim_{k\rightarrow\infty}\sum_{j=1}^\infty\tilde{\phi}_k(u_j a) = \lim_{k\rightarrow\infty}\sum_{j=1}^\infty\nu_k(\{j\}) = 0$.} 
\end{remark}

\end{document}